\documentclass[10pt]{amsart} 
\usepackage[foot]{amsaddr}
\usepackage{amsmath,amssymb,bm, color} 

 \usepackage{float}
 
\usepackage{amsmath}
\usepackage{graphicx}

\usepackage{mathrsfs}
\usepackage{bbm}
\usepackage{bm}
\usepackage{amsfonts,amssymb}
\usepackage{multirow}
\usepackage{lineno}
\usepackage{color}

\numberwithin{equation}{section}
 \newtheorem{theorem}{Theorem}[section]
 \newtheorem{lemma}[theorem]{Lemma}

\def\3bar{{|\hspace{-.02in}|\hspace{-.02in}|}}

\def\bf{{\mathbf{f}}}

\def\bn{{\mathbf{n}}}

\def\beta{{\boldsymbol{\eta}}}

\newtheorem{algorithm}{Least Squares Weak Galerkin Algorithm}[section]

\numberwithin{equation}{section}

\def\3bar{{|\hspace{-.02in}|\hspace{-.02in}|}}

\def\a#1{\begin{align*}#1\end{align*}} \def\an#1{\begin{align}#1\end{align}}

\begin{document}

\title [Least-Squares Weak Galerkin]
{A Least-Squares Weak Galerkin Method for the Biharmonic Cauchy Problem}

  \author {Chunmei Wang $\dagger$ }
  \address{Department of Mathematics, University of Florida, Gainesville, FL 32611, USA. }
  \email{chunmei.wang@ufl.edu} 
 
\author {Shangyou Zhang}
\address{Department of Mathematical Sciences,  University of Delaware, Newark, DE 19716, USA}   \email{szhang@udel.edu} 

\thanks{$\dagger$  Corresponding author. }

\begin{abstract} We develop a least-squares weak Galerkin (LS-WG) finite element method for the Cauchy problem of the biharmonic equation. The proposed approach reformulates the fourth-order equation as a coupled system of two second-order equations, which are discretized using discrete weak Laplacian operators on  weak finite element spaces. The resulting least-squares formulation yields a symmetric positive definite linear system, thereby eliminating the discrete inf-sup condition required by mixed finite element methods while avoiding the construction of globally $C^1$-conforming finite element spaces. Furthermore, the weak Galerkin framework naturally accommodates general polygonal meshes, offering considerable flexibility in mesh generation and approximation. Under the assumption that the continuous biharmonic Cauchy problem admits a unique solution, we establish the uniqueness of the discrete LS-WG solution and derive optimal-order error estimates in a discrete energy norm. Numerical experiments confirm the theoretical convergence rates and demonstrate the accuracy, robustness, and effectiveness of the proposed method.
\end{abstract}


  
\maketitle

\section{Introduction}

The biharmonic equation arises in a broad range of scientific and engineering applications, including Kirchhoff plate bending, elasticity, fluid mechanics, image processing, and phase-field models. In many practical applications, however, complete boundary information is unavailable due to inaccessible portions of the boundary or limitations in measurement devices. This leads naturally to the Cauchy problem for the biharmonic equation, in which both the function value and its normal derivative are prescribed only on an accessible portion of the boundary, while no information is available on the remaining part. Such problems arise in nondestructive testing, structural health monitoring, inverse boundary reconstruction, and data completion.

Let $\Omega\subset\mathbb{R}^2$ be a bounded convex domain with Lipschitz continuous boundary $\partial\Omega$. Suppose the boundary is partitioned into two disjoint nonempty open subsets $\Gamma_1$ and $\Gamma_2$ satisfying
\[
\partial\Omega=\overline{\Gamma}_1\cup\overline{\Gamma}_2,
\]
where $\Gamma_1$ denotes the accessible boundary and $\Gamma_2$ denotes the inaccessible boundary. We consider the biharmonic Cauchy problem of finding $u$ such that
\begin{equation}\label{model}
\begin{split}
\Delta^2u &=f,\qquad \text{in }\Omega,\\
u&=g_1, \qquad \Delta u =g_3,\qquad \text{on }\Gamma_1,\\
\nabla u\cdot\mathbf n&=g_2,\qquad  \nabla(\Delta u)\cdot\mathbf n =g_4,\qquad \text{on }\Gamma_1,
\end{split}
\end{equation}
where $f$, $g_1$, $g_2$, $g_3$ and $g_4$  are prescribed data, and $\mathbf n$ denotes the outward unit normal vector on $\partial\Omega$. No boundary conditions are imposed on $\Gamma_2$.

Unlike standard biharmonic boundary value problems, problem \eqref{model} is severely ill-posed in the sense of Hadamard. Although uniqueness may be established under suitable assumptions through unique continuation principles, the solution generally fails to depend continuously on the prescribed Cauchy data. Consequently, arbitrarily small perturbations in the boundary measurements may lead to large errors in the reconstructed solution. Compared with second-order elliptic Cauchy problems, the instability is further amplified by the fourth-order differential operator and the need to recover missing boundary information from partial measurements, making both the mathematical analysis and numerical approximation substantially more challenging.

Numerical methods for elliptic Cauchy problems have been extensively studied over the past several decades. Classical approaches include regularization techniques such as the quasi-reversibility method, Tikhonov regularization, and iterative regularization methods
\cite{engl1996regularization,Efendiev2025}, which stabilize the ill-posed problem by solving an associated well-posed or optimization problem. Boundary integral and fundamental solution methods
\cite{chen2006boundary,mcLean2000strongly} have also been successfully developed by reformulating the governing equations as boundary integral equations. While these methods are attractive for problems with relatively simple geometries, they require the accurate evaluation of singular kernels and become less effective for complex computational domains or noisy boundary data.

Finite element methods provide another important class of numerical techniques for Cauchy problems. Stabilized finite element methods, mixed finite element methods, and least-squares finite element methods have been successfully developed for second-order elliptic Cauchy problems
\cite{burman2010stabilized,bochev2005least,bramble1991least, r1, r3}. Among them, least-squares methods are particularly attractive because they naturally yield symmetric positive definite linear systems and avoid the saddle-point structure associated with mixed formulations. However, extending these methodologies to fourth-order Cauchy problems remains highly nontrivial. Direct conforming discretizations require globally $C^1$-continuous finite element spaces, such as the Argyris element, whose construction on general meshes is notoriously difficult. Mixed formulations alleviate this difficulty by introducing auxiliary variables, but typically require carefully designed finite element spaces satisfying discrete inf-sup conditions. These challenges become even more pronounced for biharmonic Cauchy problems due to their intrinsic ill-posedness.

Least-squares finite element methods are particularly well suited for ill-posed Cauchy problems because they minimize the residuals of the governing equations in a natural least-squares norm, resulting in stable approximations and symmetric positive definite linear systems. In contrast to mixed formulations, they avoid saddle-point structures and the associated discrete inf-sup conditions, leading to simpler analysis and implementation. When combined with the weak Galerkin framework, the least-squares formulation further eliminates the need for globally $C^1$
-conforming finite element spaces while naturally supporting general polygonal meshes, making it an attractive approach for fourth-order ill-posed problems.

Recently, weak Galerkin (WG) finite element methods have emerged as a powerful framework for solving high-order partial differential equations by replacing classical differential operators with weakly defined differential operators acting on weak finite element spaces. The WG methodology eliminates the need for globally smooth finite element spaces, naturally accommodates general polygonal meshes, and has been successfully applied to a wide variety of elliptic problems, including biharmonic equations
\cite{wg1,wg2,wg3,wg4,wg5,wg6,wg7,wg8,wg9,wg10,wg11,wg12,wg13,wg14,wg15,wg16,wg17,wg18,wg19,wg20,wg21,wy3655,guan, guan2}. Nevertheless, existing weak Galerkin   methods have been developed almost exclusively for well-posed boundary value problems with complete boundary conditions. To the best of our knowledge, no WG method has been proposed for the biharmonic Cauchy problem. The combination of a fourth-order differential operator, incomplete boundary data, and the intrinsic ill-posedness of the problem presents significant analytical and computational challenges that preclude a straightforward extension of existing WG methodologies. In contrast, the least-squares weak Galerkin formulation developed in this work naturally accommodates partial boundary data while preserving a symmetric positive definite linear system, making it particularly well suited for this class of ill-posed problems.

The objective of this paper is to develop a least-squares weak Galerkin (LS-WG) finite element method for the biharmonic Cauchy problem. Our starting point is the introduction of the auxiliary variable
\[
w=\Delta u,
\]
which reformulates the fourth-order problem into the coupled second-order system
\begin{equation}\label{model2}
\begin{split}
\Delta w &=f,\qquad \text{in }\Omega,\\
\Delta u-w&=0,\qquad \text{in }\Omega,\\
u&=g_1, \quad w=g_3\qquad \text{on }\Gamma_1,\\
\nabla u\cdot\mathbf n&=g_2, \quad  \nabla w\cdot \mathbf n=g_4\qquad \text{on }\Gamma_1.
\end{split}
\end{equation}
This reformulation enables the construction of a least-squares weak Galerkin formulation based on discrete weak Laplacian operators defined on discontinuous finite element spaces. The resulting formulation combines the flexibility of weak Galerkin methods with the stability of least-squares approximations. In particular, it produces a symmetric positive definite linear system, completely eliminating the discrete inf-sup condition required by mixed finite element methods while avoiding the construction of globally $C^1$-conforming finite element spaces. Moreover, the proposed method naturally supports general polygonal meshes, making it suitable for computations on complex geometries.

The main contributions of this paper are summarized as follows.

\begin{itemize}

\item We propose the first least-squares weak Galerkin finite element method for the biharmonic Cauchy problem. By reformulating the fourth-order equation as a coupled system of second-order equations, the proposed approach effectively combines least-squares stabilization with the WG method.

\item The proposed formulation yields a symmetric positive definite linear system without requiring discrete inf-sup conditions or globally $C^1$-conforming finite element spaces, while naturally accommodating general polygonal meshes. Since the resulting linear system is symmetric positive definite, efficient iterative solvers such as the conjugate gradient method together with multigrid preconditioners can be employed.

\item Under the uniqueness assumption for the continuous biharmonic Cauchy problem, we establish the uniqueness of the discrete LS-WG solution and derive optimal-order error estimates in an appropriate discrete energy norm.

\item Numerical experiments verify the theoretical convergence rates and demonstrate the accuracy, robustness, and effectiveness of the proposed method.

\end{itemize}

The remainder of the paper is organized as follows. Section~2 introduces the weak Laplacian operator and its discrete counterpart. Section~3 presents the LS-WG formulation for the biharmonic Cauchy problem and establishes the uniqueness of the discrete solution provided that the continous biharmonic Cauchy problem admits a unique solution. Section~4 derives optimal-order error estimates in the discrete energy norm. Section~5 reports numerical experiments validating the theoretical results and illustrating the performance of the proposed method.

Throughout this paper, we employ standard notation for Sobolev spaces and their associated norms. For any bounded domain $D\subset\mathbb{R}^d$ with Lipschitz boundary, we denote by $H^s(D)$ the usual Sobolev space with norm $\|\cdot\|_{s,D}$ and seminorm $|\cdot|_{s,D}$. The corresponding inner product is denoted by $(\cdot,\cdot)_{D}$ when $s=0$. For simplicity, the subscript indicating the domain is omitted whenever no confusion arises.

\section{Discrete Weak Laplacian} 

In this section, we will  introduce the definitions of the weak Laplacian and its discrete counterpart, which serve as the fundamental building blocks of the proposed LS-WG method.

Let $T$ be a polygonal element with boundary $\partial T$. A \emph{weak function} on $T$ is defined as
\[
v=\{v_0,v_b,v_n\},
\]
where $v_0\in L^2(T)$, $v_b\in L^2(\partial T)$, and $v_n\in L^2(\partial T)$. Here, $v_0$ represents the value of $v$ in the interior of $T$, $v_b$ denotes the value of $v$ on the boundary $\partial T$, and $v_n$ approximates the normal derivative $\nabla v\cdot\bn$ on $\partial T$ with $\bn$ being the unit outward normal vector on $\partial T$. In general, the boundary components $v_b$ and $v_n$ are treated as independent unknowns and are not required to coincide with the traces of $v_0$ and $\nabla v_0\cdot\bn$, respectively. The collection of all weak functions on $T$ is denoted by
\[
W(T)=\{v=\{v_0,v_b,v_n\}: v_0\in L^2(T),\;
v_b\in L^2(\partial T),\;
v_n\in L^2(\partial T)\}.
\]

The weak Laplacian, denoted by $\Delta_w$, is defined as a linear operator from $W(T)$ into the dual space of $H^2(T)$. Specifically, for any $v\in W(T)$, the weak Laplacian $\Delta_wv$ is the bounded linear functional on $H^2(T)$ satisfying
\begin{equation}\label{delta}
(\Delta_wv,\varphi)_T
=
(v_0,\Delta\varphi)_T
-
\langle v_b,\nabla\varphi\cdot\bn\rangle_{\partial T}
+
\langle v_n,\varphi\rangle_{\partial T},
\qquad
\forall\,\varphi\in H^2(T),
\end{equation}
where $\bn$ denotes the unit outward normal vector on $\partial T$.

For a nonnegative integer $r$, let $P_r(T)$ denote the space of polynomials of degree at most $r$ on $T$. The \emph{discrete weak Laplacian}, denoted by $\Delta_{w,r,T}$, is the unique polynomial in $P_r(T)$ satisfying
\begin{equation}\label{discretedelta}
(\Delta_{w,r,T}v,\varphi)_T
=
(v_0,\Delta\varphi)_T
-
\langle v_b,\nabla\varphi\cdot\bn\rangle_{\partial T}
+
\langle v_n,\varphi\rangle_{\partial T},
\qquad
\forall\,\varphi\in P_r(T).
\end{equation}

If $v_0\in H^2(T)$, an application of the usual integration by parts to \eqref{discretedelta} yields the equivalent representation
\begin{equation}\label{discretedelta2}
(\Delta_{w,r,T}v,\varphi)_T
=
(\Delta v_0,\varphi)_T
-
\langle v_b-v_0,\nabla\varphi\cdot\bn\rangle_{\partial T}
+
\langle v_n-\nabla v_0\cdot\bn,\varphi\rangle_{\partial T},  
\end{equation}
for all $\varphi\in P_r(T)$.

\section{Least-Squares Weak Galerkin Method}\label{Section:WGFEM}

Let $\mathcal{T}_h$ be a shape-regular partition of the domain
$\Omega\subset\mathbb{R}^2$ into polygonal  elements as
defined in \cite{wy3655}. Denote by $\mathcal{E}_h$ the set of all
edges  of $\mathcal{T}_h$, and let
$\mathcal{E}_h^0=\mathcal{E}_h\setminus\partial\Omega$
be the set of all interior edges. For each
$T\in\mathcal{T}_h$, let $h_T$ denote the diameter of $T$, and define
the mesh size by
$h=\max_{T\in\mathcal{T}_h}h_T$.

For an integer $k\ge 2$, we define the local weak finite element spaces
on each element $T\in\mathcal{T}_h$ by
\begin{equation*} 
W(k,T)=
\Big\{
\{v_0,v_b,v_n\}:
v_0\in P_k(T),\
v_b\in P_k(e),\
v_n\in P_{k}(e)
\Big\}.
\end{equation*}
The global weak finite element space is defined by
\begin{equation*} 
W_h=
\Big\{
\{v_0,v_b,v_n\}:
\{v_0,v_b,v_n\}|_T\in W(k,T),
\ \forall T\in\mathcal{T}_h
\Big\}.
\end{equation*}  

We further introduce the subspace
\begin{equation*} 
W_h^0=
\Big\{
\{v_0,v_b,v_n\}\in W_h:
v_b=0,\;
v_n=0
\ \text{on }\Gamma_1
\Big\}.
\end{equation*} 

For simplicity, the discrete weak Laplacian
$\Delta_{w,k,T}$  is  denoted by $\Delta_{w}$,  i.e., 
\[
(\Delta_{w} v)|_T
=
\Delta_{w,k,T}(v|_T), \qquad \forall v\in W_h,\   T\in\mathcal{T}_h.
\]

To weakly enforce the consistency between the interior and boundary
components of weak functions, we introduce the stabilizer
\[
s(u,v)
=
\sum_{T\in\mathcal{T}_h}
h_T^{-3}
\langle
u_0-u_b,\,
v_0-v_b
\rangle_{\partial T}
+
h_T^{-1}
\langle
\nabla u_0\cdot\mathbf n-u_n,\,
\nabla v_0\cdot\mathbf n-v_n
\rangle_{\partial T},
\]
for all $u,v\in W_h$.

Next, define the bilinear form
\[
a((w,u),(q,v))
=
\sum_{T\in\mathcal{T}_h} 
(\Delta_{w}w,\Delta_{w}q)_T
+
(\Delta_{w}u-w_0,\,
\Delta_{w}v-q_0)_T,
\]
for all
$(w,u),(q,v)\in W_h\times W_h$.

The least-squares weak Galerkin   method for the biharmonic
Cauchy problem \eqref{model2} is stated as follows.

\begin{algorithm}\label{PDWG1}
Find
$(w_h,u_h)\in W_h\times W_h$
such that
\begin{equation}\label{WGboun}
u_b=Q_bg_1,
u_n=Q_bg_2, w_b=Q_bg_3, w_n=Q_b g_4
\quad\text{on }\Gamma_1,
\end{equation}
and
\begin{equation}\label{WG}
a((w_h,u_h),(q,v))
+s(u_h,v)
+s(w_h,q)
=
\sum_{T\in\mathcal{T}_h}
(f,\Delta_{w}q)_T,
\end{equation}
for all
$(q,v)\in W_h^0\times W_h^0$.
Here,
$Q_b$   denote the
$L^2$ projection operators onto
$P_k(e)$.
\end{algorithm}

\begin{theorem}\label{theorem1}
Assume that the continuous biharmonic Cauchy problem \eqref{model2}
admits a unique solution. Then the least-squares weak Galerkin
scheme \eqref{WG} admits a unique solution.
\end{theorem}

\begin{proof}
It suffices to prove uniqueness. Let
$f=0$, $g_1=0$,  $g_2=0$, $g_3=0$, and $g_4=0$. Suppose
$(w_h,u_h)\in W_h\times W_h$ satisfies \eqref{WGboun}-\eqref{WG}. Choosing the test functions
$q=w_h$ and $v=u_h$ in \eqref{WG} gives
\begin{equation}\label{WG_proof}
\sum_{T\in\mathcal{T}_h}
\|\Delta_{w} w_h\|_T^2
+
\|\Delta_{w}u_h-w_0\|_T^2
+s(u_h,u_h)
+s(w_h,w_h)
=0.
\end{equation}
Hence,
\[
\Delta_{w} w_h=0,\qquad
\Delta_{w}u_h-w_0=0
\quad\text{on each }T\in\mathcal{T}_h,
\]
and
\[
u_0=u_b, \ 
\nabla u_0\cdot\mathbf n=u_n,\ 
w_0=w_b, \ 
\nabla w_0\cdot\mathbf n=w_n
\quad\text{on }\partial T.
\]

Using the identity \eqref{discretedelta2}, together with the above
relations, we obtain
\[
(\Delta_{w} u_h,\phi)_T
=(\Delta u_0,\phi)_T,
\qquad
\forall\,\phi\in P_{k}(T),
\]
which implies
\[
\Delta_{w}u_h=\Delta u_0,   
\qquad\text{on each }T\in\mathcal{T}_h.
\]
Similarly, we have
\[
    \Delta_{w}w_h=\Delta w_0,
\qquad\text{on each }T\in\mathcal{T}_h.
\]

Consequently,
\[
\Delta w_0=0,
\qquad
\Delta u_0-w_0=0, \quad\text{on each }T\in\mathcal{T}_h.
\]
Moreover, the continuity conditions imply that
$u_0\in H^2(\Omega)$ and
$w_0\in H^2(\Omega)$, and 
\[
\Delta w_0=0,
\qquad
\Delta u_0-w_0=0, \quad\text{in}\ \Omega,
\]
while the homogeneous boundary conditions yield
\[
u_0=0,  \ 
\nabla u_0\cdot\mathbf n=0,  \ w_0=0,
\ 
\nabla w_0\cdot\mathbf n=0
\quad\text{on }\Gamma_1.
\]
Since we assume the continuous Cauchy problem \eqref{model2} admits a unique
solution, it follows that
\[
u_0=0,
\qquad
w_0=0
\quad\text{in }\Omega.
\]
Therefore,
$u_b=u_n=w_b=w_n=0$, and hence
\[
u_h=0,
\qquad
w_h=0  \qquad \text{in }\Omega.
\]
This proves the uniqueness of the discrete solution.
\end{proof}

Under the uniqueness assumption for the continuous biharmonic Cauchy
problem, we define the energy norm on
$W_h^0\times W_h^0$ by
\begin{equation}\label{norm}
    \3bar(w,u)\3bar
=(
\sum_{T\in\mathcal T_h}
\|\Delta_{w} w\|_T^2
+
\|\Delta_{w}u-w_0\|_T^2
+
s(u,u)
+
s(w,w))^{1/2}.
\end{equation}

By an argument identical to that used in the proof of
Theorem~\ref{theorem1}, one readily verifies that
$\3bar\cdot\3bar$
defines a norm on
$W_h^0\times W_h^0$.
 
\section{Error Estimates}
We begin by recalling several preliminary results that will be used in the error analysis.

Let $\mathcal{T}_h$ be a shape-regular finite element partition of the domain $\Omega$. For any element $T\in\mathcal{T}_h$ and any function $\phi\in H^1(T)$, the following trace inequality holds \cite{wy3655}:
\begin{equation}\label{tracein}
\|\phi\|_{\partial T}^2
\le
C (
h_T^{-1}\|\phi\|_T^2
+
h_T\|\nabla\phi\|_T^2
  ).
\end{equation}

For each element $T\in\mathcal{T}_h$, let
$Q_0: L^2(T)  \rightarrow P_k(T)$ and 
$Q_b: L^2(e)  \rightarrow  P_k(e)$  
denote the corresponding $L^2$ projection operators.
For any $u\in H^2(\Omega)$, define the projection
$Q_h u\in W_h$ by
\[
Q_h u
=
 \{
Q_0 u,\,
Q_b u,\,
Q_b (\nabla u\cdot\mathbf n)
  \}.
\]

\begin{lemma}\cite{wy3655}
Let $\mathcal{T}_h$ be a shape-regular finite element partition of $\Omega$. Then, for any
$0\le s\le 2$ and $1\le m\le k$, the following approximation estimate holds:
\begin{equation}\label{3.1}
\sum_{T\in\mathcal{T}_h}
h_T^{2s}
\|u-Q_0 u\|_{s,T}^2
\le
Ch^{2(m+1)}
\|u\|_{m+1}^2.
\end{equation}
\end{lemma}

\begin{lemma}
The projection operators satisfy the following commutative properties:
\begin{equation}\label{EQ:CommutativeP}
\Delta_{w}(Q_h u)
=
Q_0(\Delta u),
\qquad
\forall\,u\in H^2(T),
\end{equation}
\end{lemma}

\begin{proof}
For any $q\in P_{k}(T)$, it follows from the definition of the discrete weak Laplacian \eqref{discretedelta} that
\[
\begin{aligned}
(\Delta_{w,k}Q_h u,q)_T
&=(Q_0u,\Delta q)_T
-\langle Q_bu,\nabla q\cdot\mathbf n\rangle_{\partial T}
+\langle Q_b (\nabla u\cdot\mathbf n),q\rangle_{\partial T} \\
&=(u,\Delta q)_T
-\langle u,\nabla q\cdot\mathbf n\rangle_{\partial T}
+\langle\nabla u\cdot\mathbf n,q\rangle_{\partial T} \\
&=(\Delta u,q)_T
=(\mathcal{Q}_h \Delta u,q)_T,
\end{aligned}
\]
where we have used the facts that
$\Delta q\in P_{k-2}(T)\subset P_k(T)$,
$\nabla q\cdot\mathbf n\in P_{k-1}(e)\subset P_k(e)$,
and
$q\in P_{k}(T)$.
Since the above identity holds for all $q\in P_{k}(T)$, we obtain
\[
\Delta_ {w}(Q_h u)
=
Q_0(\Delta u),
\]
which proves \eqref{EQ:CommutativeP}.  
\end{proof}

\begin{theorem}
Let $w$ and $u$ be the exact solutions to the Cauchy biharmonic problem \eqref{model2}, and let $u_h \in W_h$ and $w_h\in W_h$ be the numerical solutions to the least-squares Weak Galerkin scheme \eqref{WG}. We define the error functions  
$$e_{u_h}=\{e_{u_0}, e_{u_b}, e_{u_n}\}=Q_h  u- u_h, \qquad e_{w_h}=\{e_{w_0}, e_{w_b}, e_{w_n}\}=Q_h  w-w_h.$$ 
There  exists a constant $C$, independent of $h$, such that
\begin{equation}\label{est1}
    \3bar (e_{w_h}, e_{u_h})\3bar \leq Ch^{k-1} (\|u\|_{k+1}+ \|w\|_{k+1}).
\end{equation}
\end{theorem}
\begin{proof}
Testing the first equation in \eqref{model2} with
$\Delta_{w}q$
and the second equation with
$\Delta_ {w}v-q_0$,
followed by summation over all elements, yields
\begin{equation*}
\begin{split}
\sum_{T\in \mathcal{T}_h}(\Delta w,  \Delta_{w} q)_T+( \Delta u-w,  \Delta_{w} v-q_0)_T= \sum_{T\in \mathcal{T}_h}( f,  \Delta_{w} q)_T.
\end{split}
\end{equation*}

Applying the commutative properties
\eqref{EQ:CommutativeP},
we obtain
\begin{equation}\label{ss1}
\begin{split}
& \sum_{T\in \mathcal{T}_h}( f,  \Delta_{w} q)_T
\\
=&\sum_{T\in \mathcal{T}_h}(Q_0( \Delta w),  \Delta_{w} q)_T+( Q_0 (\Delta u-w),  \Delta_{w} v-q_0)_T\\
= & \sum_{T\in \mathcal{T}_h}(   \Delta_{w} (Q_h w),  \Delta_{w} q)_T +(   \Delta_{w}(Q_h u)-Q_0 w,  \Delta_{w} v-q_0)_T.
\end{split}
\end{equation}  

Subtracting \eqref{ss1} from \eqref{WG} gives
\begin{equation*} 
\begin{split}
 & \sum_{T\in \mathcal{T}_h}(   \Delta_{w} (Q_h w-w_h),  \Delta_{w} q)_T\\& 
 +(   \Delta_{w} (Q_h  u-u_h)-(Q_0 w-w_0),  \Delta_{w} v-q_0)_T -s(u_h, v) - s(w_h, q) = 0.
\end{split}
\end{equation*}  

Introducing the error functions
$e_{u_h}$ and $e_{w_h}$,
the above identity can be written as
\begin{equation}\label{ss}
\begin{split}
&\sum_{T\in \mathcal{T}_h}(   \Delta_{w} e_{w_h},  \Delta_{w} q)_T+(   \Delta_{w} e_{u_h}-e_{w_0},  \Delta_{w} v-q_0)_T\\&+s(e_{u_h}, v) + s(e_{w_h}, q) = s(Q_h  u, v) + s(Q_h  w, q).
\end{split}
\end{equation}  

Letting $v=e_{u_h}$ and $q=e_{w_h}$ gives
\begin{equation}\label{sss}
    \3bar (e_{w_h}, e_{u_h})\3bar ^2= s(Q_h  u, e_{u_h}) + s_2(Q_h  w, e_{w_h}).
\end{equation}
Using the trace inequality \eqref{tracein} and the estimate \eqref{3.1} with $s=0, 1, 2$ and $m=k$, we have
\begin{equation}\label{a1}
\begin{split}
& \quad \ s(Q_h  u, e_{u_h})\\ &= \sum_{T\in \mathcal{T}_h} h_T^{-3}\langle Q_0  u-Q_b  u, e_{u_0}-e_{u_b}\rangle_{\partial T}\\
& \ \ \ \ +h_T^{-1}\langle \nabla Q_0  u \cdot \bn -Q_b  ( \nabla u\cdot\bn), \nabla e_{u_0} \cdot\bn-e_{u_n}\rangle_{\partial T}\\
&\leq  (\sum_{T\in \mathcal{T}_h}h_T^{-3}\|Q_0  u-Q_b  u\|^2_{\partial T} )^{\frac{1}{2}}  (\sum_{T\in \mathcal{T}_h}h_T^{-3}\|e_{u_0}-e_{u_b}\|^2_{\partial T} )^{\frac{1}{2}}\\
& \ \ \ \ 
   + (\sum_{T\in \mathcal{T}_h}h_T^{-1}\|\nabla Q_0  u \cdot \bn -Q_b( \nabla u\cdot\bn)\|^2_{\partial T} )^{\frac{1}{2}} \\
   & \ \ \ \ 
   \cdot  (\sum_{T\in \mathcal{T}_h}h_T^{-1}\|\nabla e_{u_0} \cdot\bn-e_{u_n}\|^2_{\partial T} )^{\frac{1}{2}}\\
&\leq C (\sum_{T\in \mathcal{T}_h}h_T^{-4}\|Q_0 u- u\|^2_T + h_T^{-2}\|Q_0  u- u\|^2_{1, T})^{\frac{1}{2}}\3bar (e_{w_h}, e_{u_h})\3bar\\
 &\ \ \ \ +C(\sum_{T\in \mathcal{T}_h}h_T^{-2}\|\nabla Q_0  u \cdot \bn -  \nabla u\cdot\bn \|^2_{ T}+\|\nabla Q_0  u \cdot \bn -  \nabla u\cdot\bn \|^2_{ 1, T})^{\frac{1}{2}} \\
  &\ \ \ \ \cdot\3bar (e_{w_h}, e_{u_h})\3bar\\
&\leq C h^{k-1}\|u\|_{k+1} \3bar (e_{w_h}, e_{u_h})\3bar.
\end{split}
\end{equation}

Similar to \eqref{a1}, we have
\begin{equation}\label{a2} 
 \qquad s(Q_h  w, e_{w_h})  \leq C h^{k-1}\|w\|_{k+1} \3bar (e_{w_h}, e_{u_h})\3bar. 
\end{equation}

Combining \eqref{sss}, \eqref{a1}, and \eqref{a2} and dividing both sides by
$\3bar(e_{w_h},e_{u_h}) \3bar$
(if nonzero) yields \eqref{est1}. The proof is complete. 
\end{proof}

\begin{theorem}
Let $(w,u)$ be the exact solution of the biharmonic Cauchy problem \eqref{model2}, and let
$(w_h,u_h)\in W_h\times W_h$
be the corresponding numerical solution of the least-squares weak Galerkin scheme \eqref{WG}. Then there exists a constant $C$, independent of the mesh size $h$, such that
\[
\3bar (w-w_h, u-u_h)\3bar
\le
Ch^{k-1}(\|u\|_{k+1}
+
 \|w\|_{k+1}).
\]
\end{theorem}

\begin{proof}
 By the definition of the energy norm \eqref{norm}, the triangle inequality,  the commutative properties
\eqref{EQ:CommutativeP},  the approximation estimates
\eqref{3.1},   \eqref{a1}, \eqref{a2}, \eqref{est1}, 
 we have
\begin{equation*}
    \begin{split}
  &\3bar (w-w_h, u-u_h)\3bar^2\\  
  =& \sum_{T\in \mathcal{T}_h} \| \Delta_{w}(w-w_h)\|^2_T+\| \Delta_{w}(u-u_h)-(w-w_0)\|^2_T
 \\&+s(w-w_h, w-w_h)+ s(u-u_h, u-u_h)\\
 \leq & \sum_{T\in \mathcal{T}_h} \| \Delta_{w}(w-Q_h w)\|^2_T+\| \Delta_{w}(Q_h w-w_h)\|^2_T\\&
 +\| \Delta_{w}(u-Q_h u)-(w-Q_0 w)\|^2_T+  \| \Delta_{w}( Q_h u-u_h)-(Q_0 w-w_0)\|^2_T
 \\&+s(w-Q_h w, w-Q_h w)+ s(Q_h w-w_h,  Q_h w-w_h)\\&+s(u-Q_h u, u-Q_h u)+s(Q_h u-u_h,  Q_h u-u_h)\\    
  \leq & \sum_{T\in \mathcal{T}_h}\|\Delta_{w} (w-Q_h w)\|^2_T + \| \Delta_{w} (u-Q_h u)-(w-Q_0 w)\|^2_T\\&+s(Q_h w, Q_h w)+s(Q_h u, Q_h u)+\3bar ( Q_h w-w_h,  Q_h u-u_h)\3bar^2\\
  \leq & \sum_{T\in \mathcal{T}_h}\|\Delta   w-Q_0 \Delta w\|^2_T + \|\Delta  u-Q_0  \Delta u \|^2 +\| w-Q_0 w\|^2_T\\&+s(Q_h w, Q_h w)+s(Q_h u, Q_h u)+\3bar ( Q_h w-w_h,  Q_h u-u_h)\3bar^2\\
  \leq & Ch^{2k-2}\|w\|^2_{k+1}+Ch^{2k-2}\|u\|^2_{k+1}+Ch^{2k+2}\|w\|^2_{k+1}\\
   \leq & Ch^{2k-2}\|w\|^2_{k+1}+Ch^{2k-2}\|u\|^2_{k+1}. \end{split}
\end{equation*}
This concludes the proof. 
\end{proof}

\section{Numerical Test}

In the numerical test,  we solve the biharmonic Cauchy problem \eqref{model2} 
   on the unit square domain $\Omega=(0,1)\times(0,1)$ with
\a{ \Gamma_1 = \big[ \{0\}\times(0,1) \big] \cup \big[(0,1)\times\{0\}\big]
   \cup \big[ (0,1)\times\{1\}\big].  }
Two exact solutions of \eqref{model2} are tested:
\an{\label{s1} & \begin{cases} u = (1 - x)^4 (1 - y)^2, \\
           w = 12 (1 - x)^2 ( 1 - y)^2 + 2 ( 1 - x)^4, \end{cases} \\
\label{s2} & \begin{cases} u = 20 (5 x - 2) (2 x - 1) (5 x - 3) (y^2- y), \\
           w = 40 (25 x^2 + 75 y^2 - 25 x - 75 y + 6) (2 x - 1). \end{cases} }   
The first solution is smooth which means it varies very little inside the domain and at the free boundary
  $\partial \Omega\setminus \Gamma_1$.
In contrast, the second solution exhibits large variations at the free boundary
  $\partial \Omega\setminus \Gamma_1$, cf. Figure \ref{f-s2},
 which renders the numerical solution intractable in certain cases.

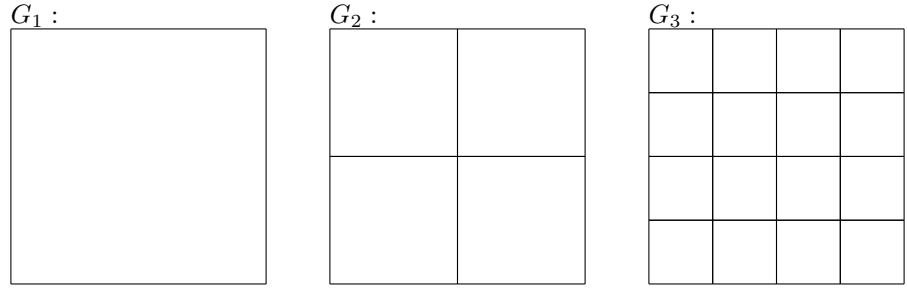
\begin{figure}[H]
\begin{center}\setlength\unitlength{2.4pt}\centering 
 \begin{picture}(140,45)(0,0) \put(0,41){$G_1:$}  \put(50,41){$G_2:$} \put(100,41){$G_3:$} 
  
\def\sq{\begin{picture}(40,40)(0,0) 
  \multiput(0,0)(40,0){2}{\line(0,1){40}}\multiput(0,0)(0,40){2}{\line(1,0){40}} \end{picture} }
  
\put(0,0){\begin{picture}(40,40)(0,0)
  \multiput(0,0)(0,40){1}{\multiput(0,0)(40,0){1}{\sq}} 
  \end{picture} }
  
\put(50,0){\setlength\unitlength{1.2pt}\begin{picture}(40,40)(0,0)
  \multiput(0,0)(0,40){2}{\multiput(0,0)(40,0){2}{\sq}} 
  \end{picture} } 
\put(100,0){\setlength\unitlength{0.6pt}\begin{picture}(40,40)(0,0)
  \multiput(0,0)(0,40){4}{\multiput(0,0)(40,0){4}{\sq}} 
  \end{picture} } 
\end{picture}\end{center}
\caption{The uniform square grids used in Tables \ref{t1}--\ref{t4}. }
\label{f-1}
\end{figure}
We apply the weak Galerkin \(P_{k}\) finite element methods for \(k = 2, 3, \text{and } 4\) on three types of grids, as shown in Figures \ref{f-1}--\ref{f-3}. The results for the solution \eqref{s1} are listed in Tables \ref{t1}--\ref{t3}. In these tables, \(G_{i}\) denotes the \(i\)-th grid of each type,
  and 
\a{ \3bar E_h\3bar=\3bar ( u- u_h, w-w_h) \3bar,
 } where $\3bar \cdot \3bar$ is defined in \eqref{norm}.

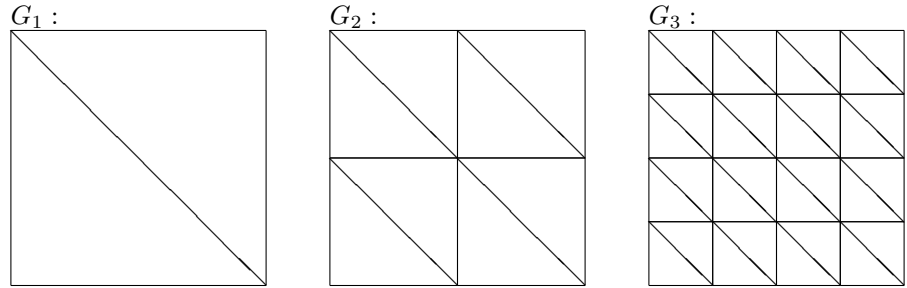
\begin{figure}[H]
\begin{center}\setlength\unitlength{2.4pt}\centering 
 \begin{picture}(140,45)(0,0) \put(0,41){$G_1:$}  \put(50,41){$G_2:$} \put(100,41){$G_3:$} 
  
\def\sq{\begin{picture}(40,40)(0,0) \put(0,40){\line(1,-1){40}}
  \multiput(0,0)(40,0){2}{\line(0,1){40}}\multiput(0,0)(0,40){2}{\line(1,0){40}} \end{picture} }
  
\put(0,0){\begin{picture}(40,40)(0,0)
  \multiput(0,0)(0,40){1}{\multiput(0,0)(40,0){1}{\sq}} 
  \end{picture} }
  
\put(50,0){\setlength\unitlength{1.2pt}\begin{picture}(40,40)(0,0)
  \multiput(0,0)(0,40){2}{\multiput(0,0)(40,0){2}{\sq}} 
  \end{picture} } 
\put(100,0){\setlength\unitlength{0.6pt}\begin{picture}(40,40)(0,0)
  \multiput(0,0)(0,40){4}{\multiput(0,0)(40,0){4}{\sq}} 
  \end{picture} } 
\end{picture}\end{center}
\caption{The triangular  grids used in Tables \ref{t1}--\ref{t4}. }
\label{f-2}
\end{figure}

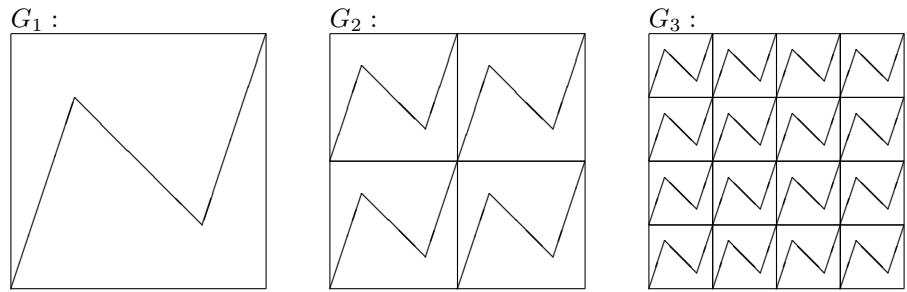
\begin{figure}[H]
\begin{center}\setlength\unitlength{2.4pt}\centering 
 \begin{picture}(140,45)(0,0) \put(0,41){$G_1:$}  \put(50,41){$G_2:$} \put(100,41){$G_3:$} 
  
\def\sq{\begin{picture}(40,40)(0,0) \put(0,0){\line(1,3){10}}  \put(40,40){\line(-1,-3){10}} \put(10,30){\line(1,-1){20}}
  \multiput(0,0)(40,0){2}{\line(0,1){40}}\multiput(0,0)(0,40){2}{\line(1,0){40}} \end{picture} }
  
\put(0,0){\begin{picture}(40,40)(0,0)
  \multiput(0,0)(0,40){1}{\multiput(0,0)(40,0){1}{\sq}} 
  \end{picture} }
  
\put(50,0){\setlength\unitlength{1.2pt}\begin{picture}(40,40)(0,0)
  \multiput(0,0)(0,40){2}{\multiput(0,0)(40,0){2}{\sq}} 
  \end{picture} } 

\put(100,0){\setlength\unitlength{0.6pt}\begin{picture}(40,40)(0,0)
  \multiput(0,0)(0,40){4}{\multiput(0,0)(40,0){4}{\sq}} 
  \end{picture} } 
\end{picture}\end{center}
\caption{The non-convex polygonal grids used in Tables \ref{t1}--\ref{t4}. }
\label{f-3}
\end{figure}

\begin{table}[ht]
  \centering  \renewcommand{\arraystretch}{1.1}
  \caption{Error profile by the $P_2$ WG element for computing \eqref{s1}. }
  \label{t1}
\begin{tabular}{c|cc|cc|cc}
\hline
 $G_i$ &\shortstack{   $\|Q_h u-u_h\|_{0}$ \\ $\|Q_h w-w_h\|_{0}$ } & $O(h^r)$ 
      &\shortstack{  $\|\Delta_w (Q_h u-u_h)\|_{0} $ \\ $\|\Delta_w (Q_h w-w_h)\|_{0} $}& $O(h^r)$  
    & $\3bar E_h\3bar $& $O(h^r)$ \\ \hline
    &  \multicolumn{6}{c}{On square meshes (Figure \ref{f-1})}    \\
\hline  
 3&    0.770E-02 &  2.9&    0.187E+01 &  1.0&    0.116E+02 &  1.1 \\
 4&    0.995E-03 &  3.0&    0.913E+00 &  1.0&    0.566E+01 &  1.0 \\
 5&    0.175E-03 &  2.5&    0.448E+00 &  1.0&    0.279E+01 &  1.0 \\
\cline{2-5} 
 3&    0.421E-01 &  2.9&    0.106E+02 &  1.1 \\
 4&    0.548E-02 &  2.9&    0.517E+01 &  1.0 \\
 5&    0.102E-02 &  2.4&    0.255E+01 &  1.0 \\
\hline  
 &  \multicolumn{6}{c}{On triangular meshes (Figure \ref{f-2}) }    \\
\hline  
 3&    0.236E-02 &  2.7&    0.134E+00 &  1.0&    0.213E+01 &  0.9 \\
 4&    0.384E-03 &  2.6&    0.676E-01 &  1.0&    0.108E+01 &  1.0 \\
 5&    0.867E-04 &  2.1&    0.339E-01 &  1.0&    0.543E+00 &  1.0 \\
\cline{2-5} 
 3&    0.142E-01 &  2.6&    0.126E+01 &  0.9 \\
 4&    0.281E-02 &  2.3&    0.639E+00 &  1.0 \\
 5&    0.701E-03 &  2.0&    0.320E+00 &  1.0 \\
\hline  
    &  \multicolumn{6}{c}{ On non-convex polygonal meshes (Figure \ref{f-3}) }    \\  
\hline   
 3&    0.582E-02 &  2.9&    0.105E+01 &  0.9&    0.851E+01 &  1.0 \\
 4&    0.774E-03 &  2.9&    0.536E+00 &  1.0&    0.426E+01 &  1.0 \\
 5&    0.130E-03 &  2.6&    0.269E+00 &  1.0&    0.213E+01 &  1.0 \\
\cline{2-5} 
 3&    0.327E-01 &  2.8&    0.608E+01 &  1.0\\
 4&    0.519E-02 &  2.7&    0.305E+01 &  1.0\\
 5&    0.866E-03 &  2.6&    0.153E+01 &  1.0\\
\hline 
    \end{tabular}%
\end{table}%

From Tables \ref{t1}–-\ref{t3}, the results match expectations. The differences between the various meshes are small. In Table \ref{t3}, at the fourth-level grid, machine accuracy is reached due to the large condition numbers of the biharmonic Cauchy problems.

\begin{table}[ht]
  \centering  \renewcommand{\arraystretch}{1.1}
  \caption{Error profile by the $P_3$ WG element for computing \eqref{s1}. }
  \label{t2}
\begin{tabular}{c|cc|cc|cc}
\hline
 $G_i$ &\shortstack{   $\|Q_h u-u_h\|_{0}$ \\ $\|Q_h w-w_h\|_{0}$ } & $O(h^r)$ 
      &\shortstack{  $\|\Delta_w (Q_h u-u_h)\|_{0} $ \\ $\|\Delta_w (Q_h w-w_h)\|_{0} $}& $O(h^r)$  
    & $\3bar E_h\3bar $& $O(h^r)$ \\ \hline
    &  \multicolumn{6}{c}{On square meshes (Figure \ref{f-1})}    \\
\hline  
 3&    0.110E-02 &  4.0&    0.962E+00 &  2.0&    0.275E+01 &  2.0 \\
 4&    0.667E-04 &  4.0&    0.234E+00 &  2.0&    0.686E+00 &  2.0 \\
 5&    0.415E-05 &  4.0&    0.576E-01 &  2.0&    0.171E+00 &  2.0 \\
\cline{2-5} 
 3&    0.254E-02 &  4.0&    0.254E+01 &  2.0\\
 4&    0.161E-03 &  4.0&    0.635E+00 &  2.0\\
 5&    0.101E-04 &  4.0&    0.159E+00 &  2.0\\
\hline  
 &  \multicolumn{6}{c}{On triangular meshes (Figure \ref{f-2}) }    \\
\hline  
 3&    0.165E-03 &  3.4&    0.271E-01 &  1.9&    0.325E+00 &  2.0 \\
 4&    0.159E-04 &  3.4&    0.690E-02 &  2.0&    0.820E-01 &  2.0 \\
 5&    0.115E-05 &  3.8&    0.173E-02 &  2.0&    0.206E-01 &  2.0 \\
\cline{2-5} 
 3&    0.141E-02 &  3.0&    0.266E+00 &  2.0\\
 4&    0.120E-03 &  3.6&    0.665E-01 &  2.0\\
 5&    0.832E-05 &  3.9&    0.166E-01 &  2.0\\
\hline  
    &  \multicolumn{6}{c}{ On non-convex polygonal meshes (Figure \ref{f-3}) }    \\  
\hline   
 3&    0.802E-03 &  3.9&    0.665E+00 &  2.0&    0.183E+01 &  2.0 \\
 4&    0.522E-04 &  3.9&    0.168E+00 &  2.0&    0.459E+00 &  2.0 \\
 5&    0.345E-05 &  3.9&    0.420E-01 &  2.0&    0.115E+00 &  2.0 \\
\cline{2-5} 
 3&    0.181E-02 &  3.9&    0.159E+01 &  2.0\\
 4&    0.129E-03 &  3.8&    0.398E+00 &  2.0\\
 5&    0.924E-05 &  3.8&    0.995E-01 &  2.0\\
\hline 
    \end{tabular}%
\end{table}%

\begin{table}[ht]
  \centering  \renewcommand{\arraystretch}{1.1}
  \caption{Error profile by the $P_4$ WG element for computing \eqref{s1}. }
  \label{t3}
\begin{tabular}{c|cc|cc|cc}
\hline
 $G_i$ &\shortstack{   $\|Q_h u-u_h\|_{0}$ \\ $\|Q_h w-w_h\|_{0}$ } & $O(h^r)$ 
      &\shortstack{  $\|\Delta_w (Q_h u-u_h)\|_{0} $ \\ $\|\Delta_w (Q_h w-w_h)\|_{0} $}& $O(h^r)$  
    & $\3bar E_h\3bar $& $O(h^r)$ \\ \hline
    &  \multicolumn{6}{c}{On square meshes (Figure \ref{f-1})}    \\
\hline  
 2&    0.281E-02 &  5.1&    0.170E+01 &  3.1&    0.171E+01 &  3.1 \\
 3&    0.888E-04 &  5.0&    0.206E+00 &  3.0&    0.207E+00 &  3.0 \\
 4&    0.284E-05 &  5.0&    0.252E-01 &  3.0&    0.254E-01 &  3.0 \\
\cline{2-5} 
 2&    0.506E-05 &  2.9&    0.365E-04 &  3.6\\
 3&    0.366E-07 &  7.1&    0.586E-06 &  6.0\\
 4&    0.995E-06 &  0.0&    0.168E-05 &  0.0\\
\hline  
 &  \multicolumn{6}{c}{On triangular meshes (Figure \ref{f-2}) }    \\
\hline  
 2&    0.151E-03 &  4.0&    0.369E-01 &  2.5&    0.413E-01 &  2.5 \\
 3&    0.884E-05 &  4.1&    0.488E-02 &  2.9&    0.551E-02 &  2.9 \\
 4&    0.351E-06 &  4.7&    0.618E-03 &  3.0&    0.700E-03 &  3.0 \\
\cline{2-5} 
 2&    0.900E-06 &  5.9&    0.111E-04 &  4.6\\
 3&    0.293E-07 &  4.9&    0.218E-06 &  5.7\\
 4&    0.333E-08 &  3.1&    0.645E-08 &  5.1\\
\hline  
    &  \multicolumn{6}{c}{ On non-convex polygonal meshes (Figure \ref{f-3}) }    \\  
\hline   
 2&    0.207E-02 &  4.9&    0.151E+01 &  2.9&    0.153E+01 &  2.9 \\
 3&    0.676E-04 &  4.9&    0.190E+00 &  3.0&    0.193E+00 &  3.0 \\
 4&    0.212E-05 &  5.0&    0.238E-01 &  3.0&    0.242E-01 &  3.0 \\
\cline{2-5} 
 2&    0.540E-05 &  6.3&    0.252E-04 &  4.3\\
 3&    0.914E-07 &  5.9&    0.509E-06 &  5.6\\
 4&    0.120E-06 &  0.0&    0.521E-06 &  0.0\\
\hline 
    \end{tabular}%
\end{table}%

In Figure \ref{f-s2}, the solutions \(u\) and \(w\) from \eqref{s2} are plotted. The free boundary is at \(x = 1.0\) at the front of the figure. We can see that \(u\) and \(w\) vary significantly near the front free boundary, especially \(w\). Thus, compared with Tables \ref{t1}--\ref{t3}, the results for solution \eqref{s2} are worse in Tables \ref{t4}--\ref{t6}. An exception occurs in Table \ref{t6} for the \(P_{4}\) weak Galerkin finite element. The reason might be that the exact solution \(w\) is inside the finite element space \(W_{h}\), while \(u_{h}\) achieves partial superconvergence. 
  
\begin{figure}[ht]
 \begin{center}\setlength\unitlength{1.0pt}
\begin{picture}(340,285)(0,0)  
  \put(0,140){\includegraphics[width=340pt]{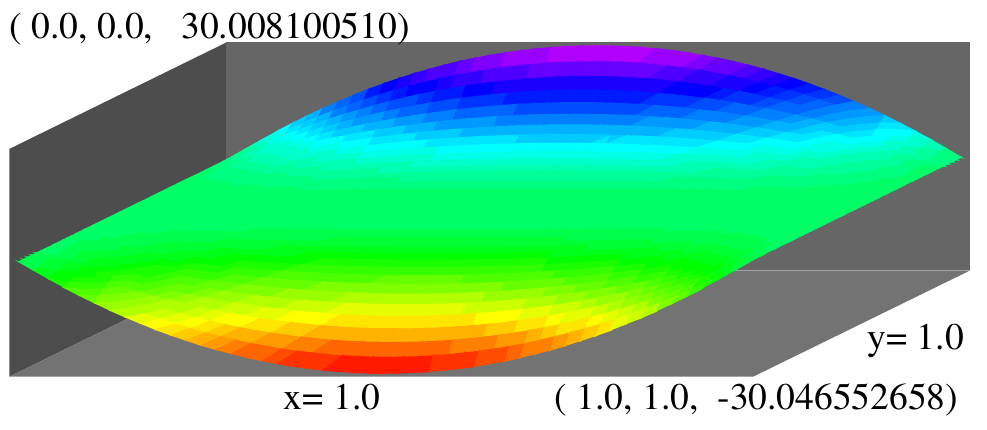}}   
  \put(0,-10){\includegraphics[width=340pt]{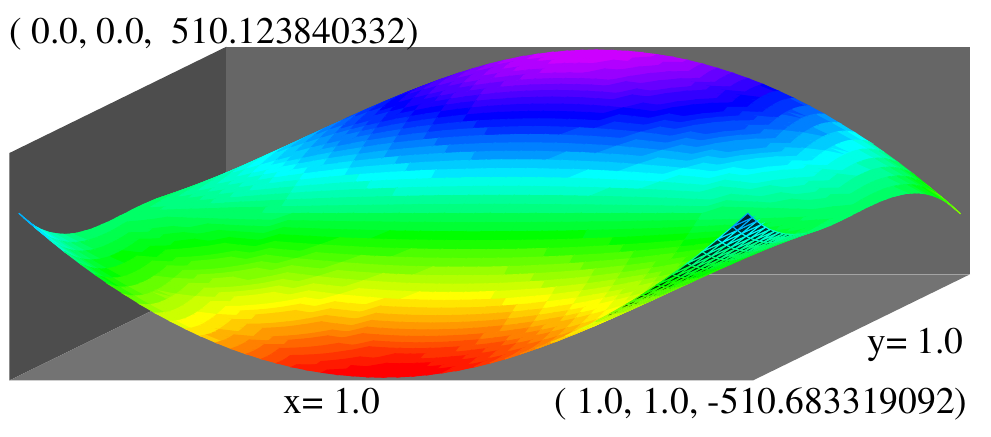}}   
 \end{picture}\end{center}
\caption{The solutions $u$ (top) and $w$ (bottom) for \eqref{s2}, where is free boundary 
   $\partial\Omega \setminus =\{x=1.0\}$. }\label{f-s2}
\end{figure}

\begin{table}[ht]
  \centering  \renewcommand{\arraystretch}{1.1}
  \caption{Error profile by the $P_2$ WG element for computing \eqref{s2}. }
  \label{t4}
\begin{tabular}{c|cc|cc|cc}
\hline
 $G_i$ &\shortstack{   $\|Q_h u-u_h\|_{0}$ \\ $\|Q_h w-w_h\|_{0}$ } & $O(h^r)$ 
      &\shortstack{  $\|\Delta_w (Q_h u-u_h)\|_{0} $ \\ $\|\Delta_w (Q_h w-w_h)\|_{0} $}& $O(h^r)$  
    & $\3bar E_h\3bar $& $O(h^r)$ \\ \hline
    &  \multicolumn{6}{c}{On square meshes (Figure \ref{f-1})}    \\
\hline  
 3&    0.936E+00 &  2.4&    0.203E+03 &  0.8&    0.272E+04 &  1.0 \\
 4&    0.155E+00 &  2.6&    0.105E+03 &  0.9&    0.136E+04 &  1.0 \\
 5&    0.229E-01 &  2.8&    0.532E+02 &  1.0&    0.681E+03 &  1.0 \\
\cline{2-5} 
 3&    0.115E+02 &  3.0&    0.245E+04 &  1.0\\
 4&    0.144E+01 &  3.0&    0.122E+04 &  1.0\\
 5&    0.181E+00 &  3.0&    0.612E+03 &  1.0\\
\hline  
 &  \multicolumn{6}{c}{On triangular meshes (Figure \ref{f-2}) }    \\
\hline  
 3&    0.709E+00 &  2.1&    0.533E+02 &  0.9&    0.117E+04 &  1.0 \\
 4&    0.135E+00 &  2.4&    0.273E+02 &  1.0&    0.583E+03 &  1.0 \\
 5&    0.224E-01 &  2.6&    0.138E+02 &  1.0&    0.291E+03 &  1.0 \\
\cline{2-5} 
 3&    0.824E+01 &  2.4&    0.792E+03 &  1.0\\
 4&    0.126E+01 &  2.7&    0.396E+03 &  1.0\\
 5&    0.296E+00 &  2.1&    0.198E+03 &  1.0\\
\hline  
    &  \multicolumn{6}{c}{ On non-convex polygonal meshes (Figure \ref{f-3}) }    \\  
\hline   
 3&    0.810E+00 &  2.4&    0.840E+02 &  0.9&    0.249E+04 &  1.0 \\
 4&    0.144E+00 &  2.5&    0.435E+02 &  0.9&    0.124E+04 &  1.0 \\
 5&    0.226E-01 &  2.7&    0.220E+02 &  1.0&    0.621E+03 &  1.0 \\
\cline{2-5} 
 3&    0.106E+02 &  2.9&    0.203E+04 &  1.0\\
 4&    0.141E+01 &  2.9&    0.102E+04 &  1.0\\
 5&    0.186E+00 &  2.9&    0.508E+03 &  1.0\\
\hline 
    \end{tabular}%
\end{table}%

\begin{table}[ht]
  \centering  \renewcommand{\arraystretch}{1.1}
  \caption{Error profile by the $P_3$ WG element for computing \eqref{s2}. }
  \label{t5}
\begin{tabular}{c|cc|cc|cc}
\hline
 $G_i$ &\shortstack{   $\|Q_h u-u_h\|_{0}$ \\ $\|Q_h w-w_h\|_{0}$ } & $O(h^r)$ 
      &\shortstack{  $\|\Delta_w (Q_h u-u_h)\|_{0} $ \\ $\|\Delta_w (Q_h w-w_h)\|_{0} $}& $O(h^r)$  
    & $\3bar E_h\3bar $& $O(h^r)$ \\ \hline
    &  \multicolumn{6}{c}{On square meshes (Figure \ref{f-1})}    \\
\hline  
 3&    0.211E+00 &  3.9&    0.194E+03 &  1.9&    0.197E+03 &  1.9 \\
 4&    0.145E-01 &  3.9&    0.492E+02 &  2.0&    0.500E+02 &  2.0 \\
 5&    0.883E-03 &  4.0&    0.123E+02 &  2.0&    0.125E+02 &  2.0 \\
\cline{2-5} 
 3&    0.340E-03 &  4.9&    0.401E-02 &  4.5\\
 4&    0.328E-04 &  3.4&    0.193E-03 &  4.4\\
 5&    0.162E-04 &  1.0&    0.432E-04 &  2.2\\
\hline  
 &  \multicolumn{6}{c}{On triangular meshes (Figure \ref{f-2}) }    \\
\hline  
 3&    0.115E+00 &  2.7&    0.193E+02 &  1.9&    0.239E+02 &  1.9 \\
 4&    0.101E-01 &  3.5&    0.491E+01 &  2.0&    0.616E+01 &  2.0 \\
 5&    0.756E-03 &  3.7&    0.123E+01 &  2.0&    0.156E+01 &  2.0 \\
\cline{2-5} 
 3&    0.880E-03 &  3.7&    0.916E-02 &  2.7\\
 4&    0.666E-04 &  3.7&    0.386E-03 &  4.6\\
 5&    0.162E-05 &  5.4&    0.107E-04 &  5.2\\
\hline  
    &  \multicolumn{6}{c}{ On non-convex polygonal meshes (Figure \ref{f-3}) }    \\  
\hline   
 3&    0.145E+00 &  3.8&    0.117E+03 &  1.9&    0.125E+03 &  1.9 \\
 4&    0.109E-01 &  3.7&    0.297E+02 &  2.0&    0.318E+02 &  2.0 \\
 5&    0.916E-03 &  3.6&    0.746E+01 &  2.0&    0.799E+01 &  2.0 \\
\cline{2-5} 
 3&    0.305E-03 &  6.1&    0.269E-02 &  4.9\\
 4&    0.525E-04 &  2.5&    0.243E-03 &  3.5\\
 5&    0.104E-03 &  0.0&    0.321E-03 &  0.0\\
\hline 
    \end{tabular}%
\end{table}%

\begin{table}[ht]
  \centering  \renewcommand{\arraystretch}{1.1}
  \caption{Error profile by the $P_4$ WG element for computing \eqref{s2}. }
  \label{t6}
\begin{tabular}{c|cc|cc|cc}
\hline
 $G_i$ &\shortstack{   $\|Q_h u-u_h\|_{0}$ \\ $\|Q_h w-w_h\|_{0}$ } & $O(h^r)$ 
      &\shortstack{  $\|\Delta_w (Q_h u-u_h)\|_{0} $ \\ $\|\Delta_w (Q_h w-w_h)\|_{0} $}& $O(h^r)$  
    & $\3bar E_h\3bar $& $O(h^r)$ \\ \hline
    &  \multicolumn{6}{c}{On square meshes (Figure \ref{f-1})}    \\
\hline  
 2&    0.281E-02 &  5.1&    0.170E+01 &  3.1&    0.171E+01 &  3.1 \\
 3&    0.888E-04 &  5.0&    0.206E+00 &  3.0&    0.207E+00 &  3.0 \\
 4&    0.284E-05 &  5.0&    0.252E-01 &  3.0&    0.254E-01 &  3.0 \\
\cline{2-5} 
 2&    0.506E-05 &  2.9&    0.365E-04 &  3.6\\
 3&    0.366E-07 &  7.1&    0.586E-06 &  6.0\\
 4&    0.995E-06 &  0.0&    0.168E-05 &  0.0\\
\hline  
 &  \multicolumn{6}{c}{On triangular meshes (Figure \ref{f-2}) }    \\
\hline  
 2&    0.134E+00 &  4.1&    0.247E+02 &  3.0&    0.274E+02 &  3.0 \\
 3&    0.720E-02 &  4.2&    0.308E+01 &  3.0&    0.344E+01 &  3.0 \\
 4&    0.790E-03 &  3.2&    0.385E+00 &  3.0&    0.431E+00 &  3.0 \\
\cline{2-5} 
 2&    0.112E-02 &  5.6&    0.888E-02 &  5.0\\
 3&    0.808E-04 &  3.8&    0.277E-03 &  5.0\\
 4&    0.496E-05 &  4.0&    0.610E-05 &  5.5\\
\hline  
    &  \multicolumn{6}{c}{ On non-convex polygonal meshes (Figure \ref{f-3}) }    \\  
\hline   
 2&    0.610E+00 &  5.0&    0.458E+03 &  3.0&    0.464E+03 &  3.0 \\
 3&    0.194E-01 &  5.0&    0.573E+02 &  3.0&    0.580E+02 &  3.0 \\
 4&    0.608E-03 &  5.0&    0.716E+01 &  3.0&    0.725E+01 &  3.0 \\
\cline{2-5} 
 2&    0.958E-03 &  6.6&    0.659E-02 &  4.9\\
 3&    0.883E-05 &  6.8&    0.954E-04 &  6.1\\
 4&    0.715E-03 &  0.0&    0.749E-03 &  0.0\\
\hline 
    \end{tabular}%
\end{table}%

\end{document}